\documentclass[12pt]{amsart}
\usepackage{float}
\restylefloat{table}
\usepackage[english]{babel}
\usepackage{amsmath}
\usepackage{graphicx}
\usepackage[margin=1in]{geometry}
\usepackage{amsfonts}
\usepackage{amssymb}
\usepackage{amsthm}
\usepackage{amssymb}
\usepackage[all]{xy}
\usepackage{color}
\usepackage{hyperref}

\newtheorem{thm}{Theorem}[section]
\newtheorem{cor}[thm]{Corollary}
\newtheorem{lem}[thm]{Lemma}

\numberwithin{equation}{section}
\theoremstyle{definition}

\newcommand{\g}{g}
\newcommand{\FF}{\mathbb F}
\newcommand{\MM}{\mathbb M}

\newcommand{\ZZ}{\mathbb Z}

\newcommand{\disc}{\operatorname {disc}}

\title{On the distribution of discriminants over a finite field}
\author[J. Chan]{Jonathan Chan}
\address[J. Chan]{Department of Mathematics, Princeton University, Princeton, NJ 08544}
\email{jdchan@princeton.edu}
\author[S. Kwon]{Soonho Kwon}
\address[S. Kwon]{Department of Mathematics, Princeton University, Princeton, NJ 08544}
\email{sskwon@princeton.edu}
\author[M. Seaman]{Michael Seaman}
\address[M. Seaman]{Department of Mathematics, Caltech, Pasadena, CA 91125}
\email{mseaman@caltech.edu}


\begin{document}

\begin{abstract}
For a prime power $q$, we show that the discriminants of monic polynomials in $\FF_q[x]$ of a fixed degree $m$ are equally distributed if 
$\gcd(q-1,m(m-1))=2$ when $q$ is odd and $\gcd(q-1,m(m-1))=1$ if $q$ is even. A theorem in the converse direction is proved when $q-1$ is squarefree.
\end{abstract}

\maketitle

\section{Introduction}\label{intro}

Let $K$ be a field and $f(x) \in K[x]$ be nonconstant of degree $m$ with leading coefficient $a_m$. If $m \ge 2$ and $f(x) = a_m\prod_{i=1}^m (x - \alpha_i)$ in a splitting field over $K$ then the \textit{discriminant} of $f$ is defined to be 
\begin{equation}\label{disc}
\disc(f) = a_m^{2m-2}\displaystyle\prod\limits_{i<j} (\alpha_i-\alpha_j)^2.
\end{equation}
When $f$ is linear, define $\disc(f)$ to be $1$. When $f$ is constant, $\disc(f)$ is not defined. 
We will be interested mostly in discriminants of monic polynomials.

When $K = \FF_q$ is a finite field with $q$ elements then 
it is natural to ask how the discriminants of the polynomials in $\FF_q[x]$ (perhaps restricted to be irreducible or to have some other fixed factorization type, as defined in Section \ref{notation}) 
are distributed among the elements of $\FF_q$. In \cite{cef} and \cite{ellenberg}, questions about squarefree polynomials in $\FF_q[x]$ and their discriminants are studied in the limit as the degree tends to $\infty$. Our focus is different, studying 
not what happens when the degree tends to $\infty$ but instead what happens in specific degrees.

Below are tables listing the number of monic polynomials of a given degree and discriminant in $\FF_q[x]$ for small values of $q$ (all taken to be prime for computational convenience). These tables were created using a program written in Java. Tables \ref{tab1}-\ref{tab4} show the distribution of the discriminants of the monic polynomials in $\FF_q[x]$ with a given degree for $q=3,5,7$, and $11$. Tables \ref{tab5}-\ref{tab8} show the distribution of the discriminants of the monic irreducible polynomials in $\FF_q[x]$ with a given degree for the same values of $q$.  Degrees run along the top row of each table and discriminant values run along the leftmost column. For example, Table \ref{tab2} tells us that 85 monic polynomials of degree 4 in $\FF_5[x]$ have discriminant 3, and Table \ref{tab6} tells us that 0 monic irreducible polynomials of degree 4 in $\FF_5[x]$ have discriminant 1 while 95 monic irreducibles of degree 4 have discriminant 2.

\begin{table}[H]
\centering
		\begin{tabular}{c|c|c|c|c|c|c|c|c|c}
Degree & $\bf{2}$ & $\bf{3}$ & $\bf{4}$& $\bf{5}$& $\bf{6}$& $\bf{7}$& $\bf{8}$ & $\bf{9}$ & $\bf{10}$\\\hline 
        \hline
0 & 3 & 9 & 27 & 81 & 243 & 729 & 2187 & 6561 & 19683\\
1 & 3 & 9 & 27 & 81 & 243 & 729 & 2187 & 6561 & 19683\\
2 & 3 & 9 & 27 & 81 & 243 & 729 & 2187 & 6561 & 19683\\
	    \end{tabular}
			\caption{Distribution of discriminants of monic polynomials in $\FF_3[x]$.}\label{tab1}
\end{table}

\begin{table}[H]
\centering
		\begin{tabular}{c|c|c|c|c|c|c|c|c|c}
Degree & $\bf{2}$ & $\bf{3}$ & 4& 5 & $\bf{6}$& $\bf{7}$& 8 & 9 & $\bf{10}$\\\hline 
        \hline
0 & 5 & 25 & 125 & 625 & 3125 & 15625 & 78125 & 390625 & 1953125 \\
1 & 5 & 25 & 95 & 475 & 3125 & 15625 & 76375 & 381875 & 1953125\\
2 & 5 & 25 & 165 & 825 & 3125 & 15625 & 72125 & 360625 & 1953125 \\
3 & 5 & 25 & 85 & 425 & 3125 & 15625 & 84125 & 420625 & 1953125\\
4 & 5 & 25 & 155 & 775 & 3125 & 15625 & 79875 & 399375 & 1953125 \\
	    \end{tabular}
			\caption{Distribution of discriminants of monic polynomials in $\FF_5[x]$.}\label{tab2}
\end{table}

\begin{table}[H]
\centering
		\begin{tabular}{c|c|c|c|c|c|c|c|c|c}
Degree & $\bf{2}$ & 3 & 4& $\bf{5}$& 6& 7& $\bf{8}$ & 9 & 10 \\\hline 
        \hline
0 & 7 & 49 & 343 & 2401 & 16807 & 117649 & 823543 &  5764801& 40353607\\
1 & 7 & 56 & 392 & 2401 & 12845 & 89915 & 823543 & 5428661& 38000627\\
2 & 7 & 14 & 98 & 2401 & 16835 & 117845 & 823543 & 6386660& 44706620 \\
3 & 7 & 21 & 147 & 2401 & 13195 & 92365 & 823543 &  6050520& 42353640 \\
4 & 7 & 77 & 539 & 2401 & 20741 & 145187 & 823543 &5479082& 38353574 \\
5 & 7 & 84 & 588 & 2401 & 15463 & 108241 & 823543 & 5142942& 36000594 \\
6 & 7 & 42 & 294 & 2401 & 21763 & 152341 & 823543 & 6100941& 42706587  \\
	    \end{tabular}
			\caption{Distribution of discriminants of monic polynomials in $\FF_7[x]$.}\label{tab3}
\end{table}

\begin{table}[H]
\centering
		\begin{tabular}{c|c|c|c|c|c|c|c|c}
Degree & $\bf{2}$ & $\bf{3}$ & $\bf{4}$ & 5& 6& $\bf{7}$ & $\bf{8}$ & $\bf{9}$  \\\hline 
        \hline
0 & 11 & 121 & 1331 & 14641 & 161051 & 1771561 & 19487171 & 214358881 \\
1 & 11 & 121 & 1331 & 13662 & 150282 & 1771561 & 19487171 & 214358881\\
2 & 11 & 121 & 1331 & 14190 & 156090 & 1771561 & 19487171 & 214358881\\
3 & 11 & 121 & 1331 & 15917 & 175087 & 1771561 & 19487171 & 214358881 \\
4 & 11 & 121 & 1331 & 14542 & 159962 & 1771561 & 19487171 & 214358881 \\
5 & 11 & 121 & 1331 & 13992 & 153912 & 1771561 & 19487171 & 214358881 \\
6 & 11 & 121 & 1331 & 15290 & 168190 & 1771561 & 19487171 & 214358881 \\
7 & 11 & 121 & 1331 & 14740 & 162140 & 1771561 & 19487171 & 214358881 \\
8 & 11 & 121 & 1331 & 13365 & 147015 & 1771561 & 19487171 & 214358881\\
9 & 11 & 121 & 1331 & 15092 & 166012 & 1771561 & 19487171 & 214358881 \\
10 & 11 & 121 & 1331 & 15620 & 171820 & 1771561 & 19487171 & 214358881\\
	    \end{tabular}
			\caption{Distribution of discriminants of monic polynomials in $\FF_{11}[x]$.}\label{tab4}
\end{table}

\begin{table}[H]
\centering
		\begin{tabular}{c|c|c|c|c|c|c|c|c|c}
Degree & $\bf{2}$ & $\bf{3}$ & $\bf{4}$& $\bf{5}$& $\bf{6}$& $\bf{7}$& $\bf{8}$ & $\bf{9}$ & $\bf{10}$ \\\hline 
        \hline
1 & 0 & 8 & 0 & 48 & 0 & 312 & 0 & 2184& 0\\
2 & 3 & 0 & 18 & 0 & 116 & 0 & 810 & 0& 5880\\
	    \end{tabular}
			\caption{Distribution of discriminants of monic irreducible polynomials in $\FF_3[x]$.}\label{tab5}
\end{table}

\begin{table}[H]
\centering
		\begin{tabular}{c|c|c|c|c|c|c|c|c|c}
Degree & $\bf{2}$ & $\bf{3}$ & 4& 5& $\bf{6}$& $\bf{7}$& 8 & 9 & $\bf{10}$ \\\hline 
        \hline
1 & 0 & 20 & 0 & 240 & 0 & 5580 & 0 & 105240 & 0\\
2 & 5 & 0 & 95 & 0 & 1290 & 0 & 22575 & 0 & 488124\\
3 & 5 & 0 & 55 & 0 & 1290 & 0 & 26175 & 0 & 488124\\
4 & 0 & 20 & 0 & 384 & 0 & 5580 & 0 & 111760 & 0\\
	    \end{tabular}
			\caption{Distribution of discriminants of monic irreducible polynomials in $\FF_5[x]$.}\label{tab6}
\end{table}

\begin{table}[H]
\centering
		\begin{tabular}{c|c|c|c|c|c|c|c}
Degree & $\bf{2}$ & 3 & 4& $\bf{5}$& 6& 7& $\bf{8}$ \\\hline
        \hline
1 & 0 & 42 & 0 & 1120 & 0 & 29952 & 0 \\
2 & 0 & 14 & 0 & 1120 & 0 & 39312 & 0 \\
3 & 7 & 0 & 84 & 0 & 5131 & 0 & 240100 \\
4 & 0 & 56 & 0 & 1120 & 0 & 48384 & 0 \\ 
5 & 7 & 0 & 336 & 0 & 6034 & 0 & 240100 \\ 
6 & 7 & 0 & 168 & 0 & 8379 & 0 & 240100 \\
	    \end{tabular}
			\caption{Distribution of discriminants of monic irreducible polynomials in $\FF_7[x]$.}\label{tab7}
\end{table}

\begin{table}[H]
\centering
		\begin{tabular}{c|c|c|c|c|c|c|c|c}
Degree & $\bf{2}$ & $\bf{3}$ & $\bf{4}$ & 5 & 6 & $\bf{7}$ & $\bf{8}$ & $\bf{9}$   \\\hline 
        \hline
1 & 0 & 88 & 0 & 6050 & 0 & 556776 & 0 & 52398808 \\
2 & 11 & 0 & 726 & 0 & 57200 & 0 & 5358606 & 0 \\
3 & 0 & 88 & 0 & 6952 & 0 & 556776 & 0  &  52398808  \\
4 & 0 & 88 & 0 & 6402 & 0 & 556776 & 0 & 52398808 \\ 
5 & 0 & 88 & 0 & 6182 & 0 & 556776 & 0 & 52398808  \\ 
6 & 11 & 0 & 726 & 0 & 61600  & 0 & 5358606 & 0 \\
7 & 11 & 0 & 726 & 0 & 59400  & 0 & 5358606 & 0 \\
8 & 11 & 0 & 726 & 0 & 53900  & 0 & 5358606 & 0 \\
9 & 0 & 88 & 0 & 6622 & 0 &  556776 & 0 & 52398808 \\
10 & 11 & 0 & 726 & 0 & 62920  & 0 & 5358606 & 0 \\
	    \end{tabular}
			\caption{Distribution of discriminants of monic irreducible polynomials in $\FF_{11}[x]$.}\label{tab8}
\end{table}

The zero entries in Tables \ref{tab5}-\ref{tab8} 
are easily explained by a theorem of Stickelberger and Swan, as we'll see after Lemma \ref{lemodd}. 
More strikingly, in some degrees all nonzero entries are equal, in which case we say discriminants are equally distributed in those degrees. 
Such degrees are in bold, {\it e.g.}, 
in Table \ref{tab6} they are degrees 2, 3, 6, 7, and 10: in these degrees the elements of $\FF_5^\times$ 
that are discriminants of monic irreducible polynomials in $\FF_5[x]$ are such discriminants an equal 
number of times.  For $\FF_7[x]$ this happens in degrees 2, 5, and 8 according to Table \ref{tab7}.  Upon closer investigation, another pattern emerges: the degrees in our data for which the discriminants of the 
monic irreducibles are equally distributed among the values they assume (half of $\FF_q^\times$) 
are the same degrees for which the discriminants of all monic polynomials are equally 
distributed in all of $\FF_q$. For instance, the degrees in Table \ref{tab2} for which the  
elements of $\FF_5$ are discriminants of the same number of monic polynomials are 2, 3, 6, 7, and 10.  Table \ref{tab3} says that in degrees 2, 5, and 8, every element of $\FF_7$ is a discriminant of the same number of monic polynomials in $\FF_7[x]$. 
If discriminants of monic polynomials in $\FF_q[x]$ of degree $m$ are equally distributed among the elements of $\FF_q$ then each element of $\FF_q$ is the discriminant of $q^{m-1}$ monic polynomials of degree $m$.

We will prove the following two theorems in order to explain and generalize 
the patterns observed above.  The first theorem is for odd $q$ and the second is for even $q$.

\begin{thm}\label{thmodd}
Let $q$ be an odd prime power. If $m \ge 2$ satisfies $\gcd(q-1,m(m-1))=2$, then the 
discriminants of monic polynomials in $\FF_q[x]$ of degree $m$ are equally distributed among the elements of $\FF_q$ and the 
discriminants of monic irreducible polynomials in $\FF_q[x]$ of degree $m$ are equally distributed among half the elements of $\FF_q^\times$.
\end{thm}

For example, if $q = 5$ and $2 \le m \le 10$ then $\gcd(q-1,m(m-1)) = 2$ only for $m = 2, 3, 6, 7$, and 10.
If $q = 7$ and $2 \le m \le 8$ then $\gcd(q-1,m(m-1)) = 2$ only for $m = 2, 5$, and 8.

\begin{thm}\label{thmeven}
Let $q$ be a power of $2$. If $m \ge 2$ satisfies $\gcd(q-1,m(m-1))=1$, then the 
discriminants of monic polynomials in $\FF_q[x]$ of degree $m$ are equally distributed among the elements of $\FF_q$ and the 
discriminants of monic irreducible polynomials in $\FF_q[x]$ of degree $m$ are equally distributed among the elements of $\FF_q^\times$. 
\end{thm}

Our paper is organized as follows.  After setting some notation and terminology in Section \ref{notation}, 
we discuss some background results in Section \ref{classical} and prove Theorems \ref{thmodd} and \ref{thmeven} in Section 
\ref{part}.  A partial converse to these theorems is in Section \ref{converse}.  Finally, in Section \ref{surj} we discuss 
the surjectivity of the discriminant for monic polynomials of a fixed degree.

\section{Notation and Terminology}\label{notation}

In this paper, $q$ is a prime power and $\FF_q$ denotes the field with $q$ elements.
Let $\MM_q$ be the set of monic polynomials in $\FF_q[x]$, and for $m \ge 1$ and $d \in \FF_q$ 
let $\MM_q(m,d)$ be the set of monic polynomials in $\FF_q[x]$ of degree $m$ with discriminant $d$.

On $\MM_q$ define the M{\"o}bius function $\mu:\MM_q \rightarrow \{-1,0,1\}$ by analogy to its definition on the positive integers:
$$
\mu(f)=
  \begin{cases}
    1 & \text{if } f \text{ is squarefree with an even number of monic irreducible factors,}\\
    -1 & \text{if } f \text{ is squarefree with an odd number of monic irreducible factors,}\\
    0 & \text{if } f \text{ is not squarefree (has a repeated irreducible factor).}\\
  \end{cases}
 $$

Let $\chi_q$ be the quadratic character of $\FF_q^\times$. If $q = p$ is prime 
it is the Legendre symbol $(\frac{\cdot}{p})$.

We will use $\pi$ to denote a monic irreducible polynomial.

For a positive integer $m$, 
let $\lambda= (\lambda_1,\lambda_2,\ldots,\lambda_k)$ be a partition of $m$: the $\lambda_i$'s are positive integers and 
$\lambda_1 + \lambda_2 + \cdots + \lambda_k = m$. 
For any field $K$ and monic 
$f \in K[x]$ of degree $m$, we say $f$ has \textit{factorization type} $\lambda$ if $f= \pi_1\pi_2\cdots\pi_k$ where the $\pi_i$'s are distinct monic irreducibles in $K[x]$ and $\deg\pi_i = \lambda_i$ for all $i$.  For example, monic irreducible polynomials of degree $m$ have factorization type $(m)$. 
We define subsets $S_{\lambda,K} \subset K[x]$ and $D_{\lambda,K} \subset K$ as follows: 
$$
S_{\lambda,K}=\left\{f \in K[x] \mid f \text{ is monic and squarefree with factorization type } \lambda \right\}
$$
and
$$
D_{\lambda,K}=\{\disc(f) \mid f \in S_{\lambda,K}\}.
$$ 
We use ``$S$'' in $S_{\lambda,K}$ 
as a reminder of ``squarefree'' (no repeated irreducible factors).  
Our interest in $S_{\lambda,K}$ and $D_{\lambda,K}$ will be in the case that $K$ is a finite field. 

Since $\FF_q[x]$ has finitely many irreducibles of any degree, there may not be enough of them 
for some factorization types.  For instance, 
if $\lambda = (1,\dots,1)$ has more than $q$ coordinates and all are 1 
then $S_{\lambda,\FF_q}$ is empty. 
On the other hand, for any $m \ge 1$ there is a monic irreducible of degree $m$ in $\FF_q[x]$, so 
$S_{(m),\FF_q}$ is not empty.

\section{Classical Results}\label{classical}

In this section we discuss some well-known results about $\FF_q[x]$ 
that we will need.

\begin{thm}[Stickelberger, Swan]\label{thmswan}
Let $q$ be an odd prime power. 
If $f \in \FF_q[x]$ is a squarefree monic polynomial of degree $m$ with $k$ distinct monic irreducible factors  then $m$ and $k$ have the same parity if and only if $\disc(f)$ is a square in $\FF_q^\times$.  Equivalently, for all monic $f$ in $\FF_q[x]$, $\chi_q(\disc(f))=(-1)^{\deg f}\mu(f)$. 
\end{thm}

\begin{proof}
See \cite[pp.~164--165]{berlekamp}. Note for squarefree monic $f$ that 
$(-1)^{\deg f}\mu(f) = (-1)^{m-k}$.
\end{proof}

In this theorem we can include $f$ that are not squarefree in the formula 
$\chi_q(\disc(f))=(-1)^{\deg f}\mu(f)$ because in that case both sides are 0.

Swan \cite[Theorem 1]{swan} proved an analogue of Theorem \ref{thmswan} when $q$ is a power of 2 using extension fields of the 2-adic numbers, 
but our work on even $q$ will not need that.
For all $q$ we will need the values of $\sum \mu(f)$ and $\sum |\mu(f)|$, where $f$ runs over monics in $\FF_q[x]$ of a fixed degree.  The next classical result computes these sums.

\begin{thm}\label{musums}
For $m \ge 2$, 
$\displaystyle \sum_{\substack{\deg f=m \\ f {\rm \ monic}}} \mu(f) = 0$ and 
$\displaystyle \sum_{\substack{\deg f=m \\ f {\rm \ monic}}} |\mu(f)| = q^m - q^{m-1}$.
\end{thm}

\begin{proof}
See \cite[pp.~14, 18]{rosen}, which uses calculations involving the zeta-function of $\FF_q[x]$. 
\end{proof}

Recall $\MM_q(m,d)$ is the set of monic polynomials of degree $m$ in $\FF_q[x]$ and discriminant $d$.

\begin{cor}\label{lema}
For any odd prime power $q$ and integer 
$m\ge 2$, 
$$
\sum_{\{d : \chi_q(d)=1\}} |\MM_q(m,d)|  = 
\sum_{\{d : \chi_q(d)=-1\}} |\MM_q(m,d)|.
$$
\end{cor}

\begin{proof}
The equation to prove can be rewritten as 
$\displaystyle 
\sum_{\substack{\deg f = m \\ f \text{ monic, } \disc(f) \not= 0}} \chi_q(\disc(f)) = 0,$
which is the same as 
$\displaystyle 
\sum_{\substack{\deg f = m \\ f \text{ monic, } \disc(f) \not= 0}} (-1)^m\mu(f) = 0$  
by Stickelberger and Swan's theorem. 
Taking the factor $(-1)^m$ out of the sum, the sum vanishes by the first result in 
Theorem \ref{musums}.
\end{proof}

\begin{cor}\label{lemb}
For any prime power $q$ and integer $m\ge 2$, the number of monic polynomials in $\FF_q[x]$ of degree $m$ and discriminant $0$ is $q^{m-1}$.
\end{cor}

\begin{proof}
The total number of monic polynomials in $\FF_q[x]$ of degree $m$ is $q^m$, so this corollary is equivalent to 
showing that the number of monic polynomials of degree $m$ in $\FF_q[x]$ with nonzero discriminant is $q^m-q^{m-1}$. Note $f$ has nonzero discriminant if and only if $|\mu(f)| = 1$, so 
the number of monic polynomials in $\FF_q[x]$ of degree $m$ with nonzero discriminant 
is $\displaystyle \sum_{\substack{\deg f = m \\ f \text{ monic, } \disc(f) \not= 0}} |\mu(f)|$, which is $q^m - q^{m-1}$ by the second result in Theorem \ref{musums}. 
\end{proof}

\section{Partitions and Equal Distribution}\label{part}

To prove Theorems \ref{thmodd} and \ref{thmeven}, 
we will use a change of variables that we first discuss 
over any field $K$. For $c \in K^{\times}$, define $\gamma_c:K[x]-\{0\} \rightarrow K[x]-\{0\}$ by 
\[\gamma_c(f(x))=c^{\deg f}f(c^{-1}x).
\]

A few straightforward properties of the change of variables $\gamma_c$ on $K[x]$ are 
\begin{itemize}
\item 
$\gamma_c(f)$ has the same degree and leading coefficient as $f$, and $\gamma_c(fg) = \gamma_c(f)\gamma_c(g)$, 
\item 
$f$ is irreducible if and only if $\gamma_c(f)$ is irreducible,
\item 
if $f$ is monic, nonconstant, and squarefree then $f$ and $\gamma_c(f)$ have the same factorization type.
\end{itemize}

Next we compare the discriminants of $f(x)$ and $\gamma_c(f(x))$.

\begin{lem}\label{discc}
If $c \in K^\times$ and $f \in K[x]$ has degree $m \ge 1$ then $\disc(\gamma_c(f)) = c^{m(m-1)}\disc(f)$. 
\end{lem}

\begin{proof}
The result is true when $m = 1$ since both sides equal 1.  Now take $m \ge 2$ and 
factor $f(x)$ as $a_m\prod_{i=1}^m (x - \alpha_i)$. 
Then $\gamma_c(f) = c^ma_m\prod_{i=1}^m (x/c - \alpha_i) = a_m\prod_{i=1}^m(x - c\alpha_i)$, so 
\begin{eqnarray*}
\disc(\gamma_c(f)) & = & a_m^{2m-2} \prod_{i < j} (c\alpha_i - c\alpha_j)^2 \\
& = & a_m^{2m-2} c^{2\binom{m}{2}}\prod_{i < j} (\alpha_i - \alpha_j)^2 \\
& = & c^{m(m-1)}\disc(f).
\end{eqnarray*}
\end{proof}

For an integer $m \ge 1$ and partition $\lambda$ of $m$, recall $S_{\lambda,K}$ is the set of monic squarefree polynomials in $K[x]$ with factorization type $\lambda$. 
If $f \in S_{\lambda,K}$ then $\gamma_c(f) \in S_{\lambda,K}$ (and the converse is true too, by applying $\gamma_{1/c}$ to $\gamma_c(f)$). Since $\gamma_c$ acts injectively as a map from $K[x]-\{0\}$ to $K[x]-\{0\}$, it is also injective as a map from $S_{\lambda,K}$ to $S_{\lambda,K}$.

The next lemma uses $\gamma_c$ to get a lower bound on the number of possible discriminants of monic squarefree polynomials in $\FF_q[x]$ 
with a common factorization type.

\begin{lem}\label{lemcard}
Let $q$ be a prime power and $m\ge 2$ be an integer. For any partition $\lambda$ 
of $m$ such that $S_{\lambda,\FF_q}$ is not empty, 
$|D_{\lambda,\FF_q}| \ge \frac{q-1}{\g}$ where $\g = \gcd(q-1,m(m-1))$. 
\end{lem}

\begin{proof}
Fix $f \in S_{\lambda,\FF_q}$ and a generator $\zeta$ of $\FF_q^{\times}$. 
For $r \in \ZZ$, $\gamma_{\zeta^r}(f) \in S_{\lambda,\FF_q}$ too and 
$\disc(\gamma_{\zeta^r}(f)) = \zeta^{m(m-1)r}\disc(f)$ by Lemma \ref{discc}.
Thus $|D_{\lambda,\FF_q}| \ge |\{\zeta^{m(m-1)r} : r \in \ZZ\}|$, and the number of different 
powers of $\zeta^{m(m-1)}$ is $(q-1)/\g$ since $\zeta$ has order $q-1$.
\end{proof}

Here is a sufficient condition for two different numbers in $\FF_q^\times$ to be 
discriminants of the same number of polynomials in $S_{\lambda,\FF_q}$.

\begin{lem}\label{lemeq}
Let $q$ be a prime power, $m\ge 2$, $\g=\gcd(q-1,m(m-1))$, and 
$\lambda$ be a partition of $m$ such that $S_{\lambda,\FF_q}$ is not empty. If $\delta_1,\delta_2 \in D_{\lambda,\FF_q}$ and $\delta_2/\delta_1$ is a $\g$th power in $\FF_q^\times$ then the number of polynomials in $S_{\lambda,\FF_q}$ with discriminant $\delta_1$ is equal to the number of polynomials in $S_{\lambda,\FF_q}$ with discriminant $\delta_2$.
\end{lem}

\begin{proof}
Write $\delta_2/\delta_1 = c^\g$ where $c \in \FF_q^\times$. 
Let $L_1=\{f \in S_{\lambda,\FF_q} \mid \disc(f)=\delta_1\}$ and $L_2=\{f \in S_{\lambda,\FF_q} \mid \disc(f)=\delta_2\}$. We want to show $|L_1| = |L_2|$. 
Let $\zeta$ generate $\FF_q^{\times}$, so $\zeta^{m(m-1)}$ has order $(q-1)/\g$, which implies 
$\{\zeta^{m(m-1)r} : r \in \ZZ\} = \{t \in \FF_q^\times : t^{(q-1)/\g} = 1\}$.
Since $(c^\g)^{(q-1)/\g} = c^{q-1} = 1$, there is an $r \in \ZZ$ such that $\zeta^{m(m-1)r}=c^\g = \delta_2/\delta_1$. 
Then for $f \in L_1$, 
$$
\disc(\gamma_{\zeta^r}(f)) = \zeta^{m(m-1)r}\disc(f) = \frac{\delta_2}{\delta_1} \delta_1=\delta_2, 
$$
so $\{\gamma_{\zeta^r}(f) \mid f \in L_1\} \subset L_2$. Since $\gamma_{\zeta^r}$ acts injectively on 
$S_{\lambda,\FF_q}$,  
$|L_1|=|\{\gamma_{\zeta^r}(f) \mid f \in L_1\}|$. Therefore $|L_1| \le |L_2|$. 
Using an identical argument with an $r$ such that $\zeta^{m(m-1)r}=c^{-\g}$, we get  
$|L_2| \le |L_1|$. Thus $|L_1|=|L_2|$, and the result follows.
\end{proof}

The next lemma gives a constraint on the discriminants of polynomials in $S_{\lambda,\FF_q}$, implying that only half the numbers in $\FF_q^\times$ could possibly be such a discriminant.

\begin{lem}\label{lemodd}
For an odd prime power $q$, integer $m \ge 2$, and partition $\lambda$ of $m$ with $k$ parts, 
the discriminant of a polynomial in $S_{\lambda,\FF_q}$ has to be a $d \in \FF_q^\times$ satisfying 
$\chi_q(d) = (-1)^{m-k}$. 
\end{lem}

\begin{proof}
This is essentially Stickelberger and Swan's theorem: any $f \in S_{\lambda,\FF_q}$ is a product of $k$ 
monic irreducibles, so if $d = \disc(f)$ then $\chi_q(d) = \chi_q(\disc f) = (-1)^{\deg f}\mu(f) = (-1)^{m-k}$.
\end{proof}

Zero entries in Tables \ref{tab5}-\ref{tab8} illustrate Lemma \ref{lemodd} when $\lambda = (m)$ with 
$k = 1$: $\chi_q(\disc(\pi)) = (-1)^{\deg \pi-1}$ for monic irreducible $\pi$ in $\FF_q[x]$. 
For example, monic irreducibles in $\FF_7[x]$ of degree 4 have discriminant with quadratic character 
$(-1)^{4-1} = -1$, which explains why Table \ref{tab7}  says there are no monic irreducibles of degree 4 in $\FF_7[x]$ with discriminant 1, 2, or 4 (the squares in $\FF_7^\times$).  Lemma \ref{lemodd} does {\it not} explain why in Table \ref{tab7} there are many monic irreducibles in $\FF_7[x]$ of degree 4 with  
each of the discriminants 3, 5, and 6.

Finally we are ready to prove Theorem \ref{thmodd}.

\begin{proof}
Fix an odd prime power $q$ and suppose $m \ge 2$ satisfies $\gcd(q-1,m(m-1))=2$.
The monic squarefree polynomials of degree $m$ in $\FF_q[x]$ are the union of the disjoint 
sets $S_{\lambda,\FF_q}$ as $\lambda$ varies over the partitions of $m$. 
In order to prove Theorem \ref{thmodd}, we will prove a theorem about 
each partition $\lambda$ of $m$: if $\gcd(q-1,m(m-1)) = 2$ and 
$S_{\lambda,\FF_q}$ is not empty, we will show 
the discriminants of the polynomials in $S_{\lambda,\FF_q}$ are equally distributed among the 
$d \in \FF_q^\times$ satisfying $\chi_q(d)=(-1)^{m-k}$, where $k$ is the number of parts in $\lambda$.

By Lemma \ref{lemcard} for $\g = 2$, $|D_{\lambda,\FF_q}| \ge (q-1)/2$. We also know by Lemma \ref{lemodd} that $D_{\lambda,\FF_q} \subset \{d \in \FF_q^\times \mid \chi_q(d)=(-1)^{m-k}\}$, and 
$|\{d \in \FF_q^\times \mid \chi_q(d)=(-1)^{m-k}\}|=(q-1)/2$. 
Thus $D_{\lambda,\FF_q} = \{d \in \FF_q^\times \mid \chi_q(d)=(-1)^{m-k}\}$.
By Lemma \ref{lemeq} with $d = 2$, for $\delta_1$ and $\delta_2$ in $D_{\lambda,\FF_q}$ such that 
$\chi_q(\delta_1/\delta_2)=1$, the number of polynomials in $S_{\lambda,\FF_q}$ with discriminant $\delta_1$ equals the number with discriminant $\delta_2$. The condition $\chi_q(\delta_1/\delta_2)=1$ holds for all $\delta_1$ and $\delta_2$ in $
D_{\lambda,\FF_q}$ since $\chi_q(\delta_1)$ and $\chi_q(\delta_2)$ equal $(-1)^{m-k}$, so 
discriminants of polynomials in 
$S_{\lambda,\FF_q}$ are equally distributed among the elements of $D_{\lambda,\FF_q}$ 
when $S_{\lambda,\FF_q}$ is not empty.

In particular, since $S_{(m),\FF_q}$ is not empty, 
when $\gcd(q-1,m(m-1)) = 2$ 
the discriminants of monic irreducibles of degree $m$ 
in $\FF_q[x]$ are equally distributed among all $d \in \FF_q^\times$ satisfying $\chi_q(d) = (-1)^{m-1}$. 
That proves the part of Theorem \ref{thmodd} about irreducible polynomials of degree $m$ and 
explains the columns in Tables \ref{tab5}-\ref{tab8} in which all nonzero entries are equal.

To prove the part of Theorem \ref{thmodd} about all monic polynomials of degree $m$, we use all partitions of $m$.
By Lemma \ref{lemodd}, 
when $\lambda$ runs over partitions of $m$ and the number of parts of $\lambda$ is denoted $\ell(\lambda)$, 
 \begin{equation}\label{union1}
 \bigcup_{\{\lambda : 2 \mid (m-\ell(\lambda))\}} S_{\lambda,\FF_q}=\{f \in \MM_q \mid \deg f=m, \chi_q(\disc(f))=1\}
 \end{equation}
 and 
 \begin{equation}\label{union2}
 \bigcup_{\{\lambda : 2 \nmid (m-\ell(\lambda))\}} S_{\lambda,\FF_q}=\{f \in \MM_q \mid \deg f=m, \chi_q(\disc(f))=-1\}.
 \end{equation}
 By the first part of this proof, the discriminants of the $f$ in  
 (\ref{union1}) are equally distributed among the elements of $\FF_q^{\times}$ with quadratic character $1$, and the discriminants of the $f$ in 
 (\ref{union2}) are equally distributed among the elements of $\FF_q^{\times}$ with quadratic character $-1$. 
 The size of (\ref{union1}) is 
 $\sum_{\{d : \chi_q(d) = 1\}} |\MM_q(m,d)|$ 
 and the size of (\ref{union2}) is 
 $\sum_{\{d : \chi_q(d) = -1\}} |\MM_q(m,d)|$, 
 and these are equal by Corollary \ref{lema}. Thus 
 discriminants of {\it squarefree}  monic polynomials of degree $m$ in $\FF_q[x]$ are equally distributed among the elements of  $\FF_q^{\times}$. 
 By the proof of Corollary \ref{lemb} there are $q^m-q^{m-1}$ monic polynomials in $\FF_q[x]$ of degree $m$ with nonzero discriminant. 
 We just showed the discriminants of these polynomials are equally distributed in $\FF_q^\times$, so  there are $\frac{q^m-q^{m-1}}{q-1}=q^{m-1}$ monic polynomials of degree $m$ with discriminant $d$ for each $d \in \FF_q^{\times}$. 
By Corollary \ref{lemb}, $q^{m-1}$ monic polynomials in $\FF_q[x]$ of degree $m$ have discriminant $0$. 
Thus for all $d \in \FF_q$, $q^{m-1}$ monic polynomials in $\FF_q[x]$ of  degree $m$ have discriminant $d$. 
\end{proof}

We now prove Theorem \ref{thmeven}. 

\begin{proof}
Let $q$ be a power of 2 and suppose $m \ge 2$ satisfies $\gcd(q-1,m(m-1))=1$. We will prove that for each partition $\lambda$ of $m$ such that $S_{\lambda,\FF_q}$ is not empty, 
the discriminants of the polynomials in $S_{\lambda,\FF_q}$ are equally distributed among the elements of $\FF_q^\times$. 

By Lemma \ref{lemcard} for $\g = 1$, $|D_{\lambda,\FF_q}| \ge q-1$. Since $D_{\lambda,\FF_q} \subset \FF_q^{\times}$ we get $D_{\lambda,\FF_q}=\FF_q^{\times}$. Then 
by Lemma \ref{lemeq} with $\g = 1$, for $\delta_1$ and $\delta_2$ in $\FF_q^\times$ 
the number of polynomials in $S_{\lambda,\FF_q}$ with discriminant $\delta_1$ equals the number with discriminant $\delta_2$. Thus discriminants of polynomials 
in $S_{\lambda,\FF_q}$ are equally distributed among the elements of $\FF_q^\times$.  In particular, 
since $S_{(m),\FF_q}$ is not empty, we have proved that if $\gcd(q-1,m(m-1)) = 1$ then 
discriminants of monic irreducibles of degree $m$ in $\FF_q[x]$ are equally distributed among the elements of $\FF_q^\times$.  That proves the part of Theorem \ref{thmeven} about the distribution of discriminants of monic irreducible polynomials in $\FF_q[x]$. 

Since discriminants of polynomials in $S_{\lambda,\FF_q}$ are equally distributed among the elements of $\FF_q^\times$ if $S_{\lambda,\FF_q}$ is not 
empty, by letting $\lambda$ vary over all partitions of $m$ we get that 
discriminants of squarefree monic polynomials of degree $m$ in $\FF_q[x]$ are equally distributed among the elements of $\FF_q^\times$. 
Deducing from this that, for each $d \in \FF_q$, $q^{m-1}$ monic polynomials of degree $m$ in $\FF_q[x]$ have discriminant $d$  is identical to 
the end of the proof of Theorem \ref{thmodd}.
\end{proof}

\begin{cor}\label{cora}
For each prime power $q$, the hypotheses of Theorem  
$\ref{thmodd}$ and Theorem $\ref{thmeven}$ are each satisfied for infinitely many $m \ge 2$. 
\end{cor}

\begin{proof} 
For each integer $a \ge 3$, set $m_a = a(q-1) - 1$, so $m_a \ge a-1\ge 2$.  Then 
$$
m_a(m_a-1) \equiv (-1)(-2) \equiv 2 \bmod q-1, 
$$
so $\gcd(q-1,m_a(m_a-1)) = \gcd(q-1,2)$, which is 2 if $q$ is odd and is 1 if $q$ is even. 
\end{proof}

\section{The Converse Direction}\label{converse}

Our numerical data suggest that the converse of Theorem \ref{thmodd} might be true: equal distribution of discriminant values in degree $m$ 
should imply $\gcd(q-1,m(m-1)) = 2$ for odd $q$.  While we have not been able to prove this, we will prove a partial converse to 
the stronger version of Theorem \ref{thmodd} that we actually proved: if $\gcd(q-1,m(m-1))=2$, then for every partition $\lambda$ of $m$, the discriminants of the polynomials in $S_{\lambda,\FF_q}$ are equally distributed among the elements of $\{c \in \FF_q \mid \chi_q(c)=(-1)^{m-\ell(\lambda)}\}$. 
The converse of this stronger fact will be proved when $q-1$ is squarefree, 
and the same method will carry over to even $q$ when $q-1$ is squarefree.

\begin{thm}\label{convodd}
Suppose $q$ is an odd prime power and $q-1$ is squarefree. If $m \ge 2$ and for every partition $\lambda$ of $m$ the discriminants of the polynomials in $S_{\lambda,\FF_q}$ are equally distributed among the elements of $\{c \in \FF_q^{\times} \mid \chi_q(c)=(-1)^{m-\ell(\lambda})\}$, then $\gcd(q-1,m(m-1))=2$.
\end{thm}

To prove this, we introduce some notation.  For a prime $\ell$ and integer $n$, let $v_\ell(n)$ be the exponent of the highest power of $\ell$ dividing $n$.
For $m \ge 1$ let $N_q(m)$ be the number of monic irreducibles in $\FF_q[x]$ of degree $m$. 
Gauss showed 
\begin{equation}\label{muformula}
N_q(m)=\frac{1}{m}\left(\displaystyle\sum\limits_{d \mid m} q^d \mu(m/d)\right).
\end{equation}
for all $m$ \cite[pp.~287-289]{jacobson}.

\begin{lem}\label{vlemma}
If $q$ is a prime power and $\ell$ is an odd prime dividing $q-1$, then $v_\ell(N_q(2^t\ell)) = v_\ell(q-1)-1$ for all nonnegative integers $t$.
\end{lem}

\begin{proof}
When $q$ is a prime power and $\ell$ is an odd prime dividing $q-1$, \eqref{muformula} implies 
$$
N_q(\ell)=\frac{q^\ell-q}{\ell}
$$
and for $t \ge 1$ 
$$
N_q(2^t\ell)=\frac{q^{2^t\ell}-q^{2^{t-1}\ell}-q^{2^t}+q^{2^{t-1}}}{2^t\ell}.
$$

Suppose $t=0$. Then
$$
v_\ell(N_q(2^t\ell))=v_\ell(N_q(\ell))=v_\ell(q^\ell-q)-1=v_\ell(q^{\ell-1}-1)-1.
$$
It is standard for odd primes $\ell$ that
\begin{equation}\label{an1}
a \equiv 1 \bmod \ell \Longrightarrow v_\ell(a^n-1)=v_\ell(a-1)+v_\ell(n).
\end{equation}
Therefore $v_\ell(q^{\ell-1}-1)=v_\ell(q-1)$ if $\ell \mid (q-1)$, so $v_\ell(N_q(\ell))=v_\ell(q-1)-1$.

Next suppose $t \ge 1$. Then
\begin{equation}\label{vell}
v_\ell(N_q(2^t\ell))=v_\ell(q^{2^t\ell}-q^{2^{t-1}\ell}-q^{2^t}+q^{2^{t-1}})-1.
\end{equation}
We have
$$
q^{2^t\ell}-q^{2^{t-1}\ell}-q^{2^t}+q^{2^{t-1}} =q^{2^{t-1}\ell}(q^{2^{t-1}\ell}-1)-q^{2^{t-1}}(q^{2^{t-1}}-1).
$$
By (\ref{an1}), if $\ell \mid (q-1)$ then 
$$
v_\ell(q^{2^{t-1}\ell}(q^{2^{t-1}\ell}-1))=v_\ell(q^{2^{t-1}\ell}-1)=v_\ell(q-1)+1
$$
and
$$
v_\ell(q^{2^{t-1}}(q^{2^{t-1}}-1))=v_\ell(q^{2^{t-1}}-1)=v_\ell(q-1).
$$
Therefore
$$
v_\ell(q^{2^{t-1}\ell}(q^{2^{t-1}\ell}-1)-q^{2^{t-1}}(q^{2^{t-1}}-1))=v_\ell(q-1), 
$$
so $v_\ell(N_q(2^t\ell))=v_\ell(q-1)-1$ from (\ref{vell}).
\end{proof}

We now prove Theorem \ref{convodd}.

\begin{proof}
We will prove the contrapositive. That is, if $q$ is an odd prime, $q-1$ is squarefree, and $\gcd(q-1,m(m-1)) \neq 2$, then we will prove there is a 
partition $\lambda$ of $m$ such that the discriminants of the polynomials in $S_{\lambda,\FF_q}$ are not equally distributed among the elements of $\{c \in \FF_q^{\times} \mid \chi_q(c)=(-1)^{m-\ell(\lambda})\}$ because $\frac{q-1}{2} \nmid |S_{\lambda,\FF_q}|$.

If $\lambda=(\lambda_1,\lambda_2,\ldots,\lambda_k)$ and no two of the $\lambda_i$ are equal, then
\[|S_{\lambda,\FF_q}|=\displaystyle\prod\limits_{i=1}^k N_q(\lambda_i),
\]
where $N_q$ is defined after the statement of Theorem \ref{convodd}. 

Since $q-1$ and $m(m-1)$ are even and $q-1$ is squarefree, 
that $\gcd(q-1,m(m-1)) \neq 2$ is equivalent to the following statement: there exists an odd prime $\ell$ such that $\ell \mid (q-1)$ and either 
$\ell \mid m$ or $\ell \mid (m-1)$. For such an odd prime $\ell$, 
let $2^{a_1}+2^{a_2}+\cdots+2^{a_k}$ be the base-$2$ expansion of $m/\ell$ if $\ell \mid m$ or of $(m-1)/\ell$ if $\ell \mid (m-1)$. 
The exponents $a_i$ are distinct. 

We will first deal with the case that 
$\ell \mid m$. Consider the partition $\lambda=(2^{a_1}\ell,2^{a_2}\ell,\ldots,2^{a_k}\ell)$ of $m$. Since the $a_i$ are distinct, we have
$$
|S_{\lambda,\FF_q}|=\displaystyle\prod\limits_{i=1}^k N_q(2^{a_i}\ell).
$$

By Lemma \ref{vlemma}, $v_\ell(N_q(2^{a_i}\ell))=v_\ell(q-1)-1$ for all $i$. Since $q-1$ is squarefree we have $v_\ell(q-1) = 1$, so 
$\ell \nmid N_q(2^{a_i}\ell)$ for all $i$. Therefore $\ell \nmid |S_{\lambda,\FF_q}|$, so $\frac{q-1}{2} \nmid |S_{\lambda,\FF_q}|$.

In the case that $\ell \mid (m-1)$, 
use the partition $\lambda=(2^{a_1}\ell,2^{a_2}\ell,\ldots,2^{a_k}\ell,1)$ of $m$. Since $N_q(1) = q$ is not divisible by $\ell$, 
we get $\frac{q-1}{2} \nmid |S_{\lambda,\FF_q}|$ in the same way. 
\end{proof}

Here is a similar result when $q$ is even.

\begin{thm}\label{conveven}
Suppose $q$ is an even prime power and $q-1$ is squarefree. If $m \geq 2$ and for every partition $\lambda$ of $m$ the discriminants of the polynomials in $S_{\lambda,\FF_q}$ are equally distributed among the elements of $\FF_q^{\times}$, then $\gcd(q-1,m(m-1))=1$.
\end{thm}

\begin{proof}
Like the previous result, we prove Theorem \ref{conveven} by proving its contrapositive. That is, if $q$ is a power of $2$, $q-1$ is squarefree, and 
$\gcd(q-1,m(m-1)) \neq 1$, then we will construct a partition $\lambda$ of $m$ such that the discriminants of the polynomials in $S_{\lambda,\FF_q}$ are not equally distributed among the elements of $\FF_q^\times$ because $q-1 \nmid |S_{\lambda,\FF_q}|$.

Since $q-1$ is odd, the condition $\gcd(q-1,m(m-1)) \neq 1$ is equivalent to the following statement: there exists an 
odd prime $\ell$ such that $\ell \mid (q-1)$ and either $\ell \mid m$ or $\ell \mid (m-1)$.  Using such an odd prime $\ell$, 
in exactly the same way as in the proof of Theorem \ref{convodd} build a partition $\lambda$ of $m$ and show 
$|S_{\lambda,\FF_q}|$ is not divisible by $\ell$, so $|S_{\lambda,\FF_q}|$ is not divisible by $q-1$.
\end{proof}

We do not think $q-1$ being squarefree is truly necessary for such converses, but it is only used to make the proofs above work.
It would be good to eliminate this hypothesis and establish a direct converse to Theorems \ref{thmodd} and \ref{thmeven}.

\section{Surjectivity of the Discriminant}\label{surj}

Theorem \ref{thmswan} gives a constraint on discriminants of monic polynomials in $\FF_q[x]$ when we fix 
the factorization type. When we don't fix the factorization type we do not think 
there is a constraint on discriminants anymore: 
for each $m \ge 2$ we believe 
discriminants of monic polynomials of degree $m$ in $\FF_q[x]$ should range over all of $\FF_q$.  
This is trivial when $m = 2$. 
We will prove it for $m$ greater than or equal to the characteristic of $\FF_q$.

\begin{thm}
\label{thmsur}
Let $p$ be a prime and let $q$ be a power of $p$. For $m \ge p$ and $d \in \FF_q$, there is a monic polynomial $f \in \FF_q[x]$ of degree $m$ such that $\disc(f)=d$.
\end{thm}

\begin{proof}
Our proof will be by explicit construction of suitable polynomials and has four cases.  

If $f(x) = a_m\prod_{i=1}^m (x- \alpha_i)$ then a standard formula for the discriminant of $f(x)$, obtained from rearranging terms in (\ref{disc}), is 
\begin{equation}\label{disc2}
\disc(f)=(-1)^{m(m-1)}a_m^{m-2}\displaystyle\prod\limits_{i=1}^m f'(\alpha_i).
\end{equation}

\underline{Case 1}: If $p\nmid m$, let $f_a(x)=x^m-x^{m-p}+ax^p+1$ with $a \in \FF_q$. 
We have 
$$
f_a'(x)=mx^{m-p-1}(x^p-1)=mx^{m-p-1}(x-1)^p.
$$
Factoring $f_a(x)$ as $\prod_{i=1}^m (x- \alpha_i)$, so $\prod_{i=1}^m \alpha_i = (-1)^m$, 
by \eqref{disc2} we have
$$
\disc(f_a) =(-1)^{m(m-1)/2}m^m(a+1)^p.
$$
The factor $m^m$ is nonzero in $\FF_q$ since $p\nmid m$ and 
there are no restrictions on $a$. Since the $p$th power map on $\FF_q$ is surjective, letting $a$ run  over all of $\FF_q$ 
makes $\disc(f_a)$ run over all of $\FF_q$. 

\underline{Case 2}: If $p \not= 2$ and $p\mid m$, let $f_a(x)=x^m+ x^2 +a$ with $a \in \FF_q$. 
Then $\disc(f_a) =  (-1)^{m(m-1)/2}(-2)^ma$. 
Since $-2 \not= 0$ in $\FF_q$, $\disc(f_a)$ runs over $\FF_q$ as $a$ does.  

 \underline{Case 3}: If $p=2$ and $p\mid m$ with $m \ge  4$, 
 let $f_a(x)=x^m+x^3 +a$ with $a \in \FF_q$. 
Then $\disc(f_a)  = a^2$. 
Squaring on $\FF_q$ is surjective since $p = 2$, so $\disc(f_a)$ runs over $\FF_q$ as $a$ does. 

\underline{Case 4}:  If $p=2$ and $m = 2$, let $f_a(x) = x^2 + ax + 1$ with $a \in \FF_q$. 
Then $\disc(f_a) = a^2$, which runs over $\FF_q$ as $a$ does.
\end{proof}

\textit{Acknowledgments}. This research was funded by the Clay Mathematics Institute and the PROMYS Foundation. We would like to thank 
Keith Conrad for proposing the topic and providing invaluable mentorship throughout the course of the research, 
Zhaorong Jin for guidance through the early stages of our work, and 
Glenn Stevens and the PROMYS program for providing the facilities where the research was performed.

\end{document}